\documentclass[12pt,twoside,reqno]{amsart}
\usepackage{enumerate}
\usepackage[latin1]{inputenc}
\usepackage[OT1]{fontenc}

\theoremstyle{plain}
\newtheorem{thm}{Theorem}[section]

\newtheorem{proposition}[thm]{Proposition}
\newtheorem{lem}[thm]{Lemma}

\theoremstyle{remark}
\newtheorem{rem}[thm]{Remark}
\newtheorem{defin}[thm]{Definition}
\newtheorem{question}[thm]{Question}

\newtheorem{exa}[thm]{Example}

\newcommand{\R}{\ensuremath{\mathbb{R}}}

\newcommand{\N}{\ensuremath{\mathbb{N}}}
\newcommand{\Ll}{\ensuremath{\mathcal{L}}}

\newcommand{\HH}{\mathcal{H}}

\DeclareMathOperator{\diam}{diam}

\begin{document}

\title{Thin and fat sets for doubling measures in metric spaces}

\author{Tuomo Ojala}
\author{Tapio Rajala}
\author{Ville Suomala}

\address{Department of Mathematics and Statistics \\
         P.O. Box 35 (MaD) \\
         FI-40014 University of Jyv\"askyl\"a \\
         Finland}
\email{tuomo.j.ojala@jyu.fi}

\address{Department of Mathematical Sciences \\
         P.O. Box 3000 \\
         FI-90014 University of Oulu \\
         Finland}
\email{ville.suomala@oulu.fi}

\address{Scuola Normale Superiore\\
Piazza dei Cavalieri 7\\
I-56127 Pisa\\ Italy}
\email{tapio.rajala@sns.it}

\thanks{
T. Ojala was supported by the Vilho, Yrj\"o and Kalle V\"ais\"al\"a fund,
T. Rajala by the European Project ERC AdG *GeMeThNES* and
V. Suomala by the Academy of Finland project \#126976.
}
\date{\today}


\begin{abstract}
 We consider sets in uniformly perfect metric spaces which are null for every doubling measure
 of the space or which have positive measure for all doubling measures. These sets are called
 thin and fat, respectively. In our main results, we give sufficient
 conditions for certain cut-out sets being thin or fat.
\end{abstract}


\maketitle

\section{Introduction}
In this paper, we study the size of subsets of a metric space from the
point of view of doubling measures.
We say that a subset $A$ of a metric space $X$ is \emph{fat}
if it has positive measure for all doubling measures of $X$ and
\emph{thin}, if it has zero measure with respect to all the doubling
measures. Recall that a (Borel regular outer-) measure $\mu$ on a metric space $(X,d)$
is called \emph{doubling (with constant $C$)} if there exists a
constant $C \ge 1$ depending only on $\mu$ so that
\begin{equation*}
 0<\mu(B(x,2r)) \leq C \mu(B(x,r))<\infty,
\end{equation*}
for all $x\in X$ and $r>0$. By $B(x,r)$ we mean the closed ball centered at $x$ with radius $r$.
The notation $U(x,r)$ is used for the corresponding open ball.
We denote the collection of all
doubling measures on the
space $X$ by $\mathcal{D}(X)$. We recall the following closely related
concept: A metric space $X$ is \emph{a doubling metric space}, if
there is $N\in\N$ such that any ball of radius $r$ may be covered by
a collection of $N$ balls of radius $r/2$. It is easy to see by a
simple volume argument, that $ \mathcal{D}(X)\neq\emptyset$ implies
that $X$ is doubling. On the other hand, if the space is complete and
doubling, then it can be shown that $\mathcal{D}(X)\neq\emptyset$, and
typically the collection $\mathcal{D}(X)$ is rather rich. See \cite{VolbergKonyagin1987},
 \cite{LuukkainenSaksman1998}, \cite{Wu1998}, \cite{KaeSuoRaj}.

Our aim is to find geometric conditions for subsets of metric spaces
which guarantee that the set is fat or thin.
In order to obtain such results,
we need to require mild regularity from the space:
A metric space $X$ is called \emph{uniformly perfect (with constant $D$)},
if it is not a singleton, and if there exists a constant $D \geq 1$
such that
\[
X \setminus B(x,r) \neq \emptyset \Longrightarrow B(x,r) \setminus B(x,r/D) \neq \emptyset,
\]
for all $x\in X$ and $r>0$.

The main result of this paper, Theorem \ref{thm:gradulause}, deals with
so called \emph{cut-out sets} that are formed by removing a countable
collection of closed balls from a uniformly perfect metric space. The
result states that if the diameter of the removed balls decays
sufficiently fast, then the remaining cut-out set is always
either fat or thin (or both if there are no doubling measures supported by the metric space).
In order to formulate what we mean by fast enough, we employ for all $p>0$ the notation
\[
 \ell^p := \left\{(\alpha_n)_{n=1}^{\infty} ~:~ 0<\alpha_n<1 \text{ and }\sum_{n=1}^{\infty} \alpha_n^p < \infty\right\}.
\]
For $0<p_1 \le p_2$ , and $0<\alpha<1$,
we have $\alpha^{p_1} \ge \alpha^{p_2}$ and therefore
$\ell^{p_1} \subset \ell^{p_2}$.
With this in mind, we introduce the abbreviation
\[
 \ell^0 := \bigcap_{p>0} \ell^p.
\]

\begin{thm}\label{thm:gradulause}
 Let $(B_i)_{i=1}^\infty$ be a sequence of closed balls in a uniformly perfect metric space $(X,d)$ so that
 $(\diam(B_i))_{i=1}^\infty \in \ell^0$. Then the set
 \[
  E:=X \setminus \bigcup_{i=1}^\infty B_i
 \]
 is fat or thin.
\end{thm}

We will prove Theorem \ref{thm:gradulause} in Section \ref{section:mainthm}.
It is important that the removed balls are closed. In fact, Theorem
\ref{thm:gradulause} fails for open balls. This will be seen in
Example \ref{ex:openfails}. The result also fails, if we remove the
assumption on uniform perfectness, see Example \ref{ex:nuperf}.

A slightly weaker version of Theorem \ref{thm:gradulause} was proved by
Staples and Ward, \cite[Theorem 1.2]{StaplesWard}, in the case
$X=\R$. They proved that the cut-out set $E$ is fat if it has positive
Lebesgue measure and if
$(\diam(B_i))\in\ell^0$. Staples and Ward formulated their result using
quasisymmetric functions. Thin and fat sets on the real line were also
discussed in \cite{Wu} and \cite{BuHaMac},
again in terms of quasisymmetric functions.

Let us next review the connection between doubling measures and
quasisymmetric mappings. Recall first that a homeomorphism $f$
between two metric spaces $(X,d_X)$ and $(Y,d_Y)$ is \emph{quasisymmetric} if there exists a homeomorphism 
$\eta \colon [0,\infty) \to [0,\infty)$ so that for any three distinct points $x,y,z \in X$ we have
\[
 \frac{d_Y(f(x),f(y))}{d_Y(f(x),f(z))} \le \eta\left(\frac{d_X(x,y)}{d_X(x,z)}\right).
\]
If $\mu\in\mathcal{D}(\R)$, then a function $f$ satisfying for all
$x\in\R$ the identity
\[
 f(x)-f(0) = \int_0^x d\mu
\]
is quasisymmetric. Conversely, for a quasisymmetric function
$f\colon\R\rightarrow\R$, the measure $\mu$ defined via
$\mu([a,b]) = f(b)-f(a)$ (or $f(a)-f(b)$ if $f$ is decreasing), for $a<b$, is doubling.

In the setting of uniformly perfect metric spaces, we follow the terminology introduced
by Heinonen in \cite{Heinonen} and call a set $E \subset X$ \emph{quasisymmetrically thick},
if $\HH^q(f(E))>0$ whenever $f\colon X \to Y$ is a quasisymmetric homeomorphism from $X$
onto an Ahlfors $q$-regular space $Y$.
We can also define
\emph{quasisymmetrically null} sets as those $E\subset X$ for which
$\HH^q(f(E))=0$ whenever $f\colon X \to Y$ is a quasisymmetric homeomorphism from $X$
onto an Ahlfors $q$-regular space $Y$.

In both of the previous definitions $q$ is allowed to depend on the space $Y$.

By examining the proof of \cite[Proposition 14.41]{Heinonen} we notice an equivalence between
the quasisymmetrically thick/null sets and fat/thin sets in uniformly perfect spaces.
We include the proof of this fact here. Notice that we do not
need to assume the space $(X,d)$ to be complete.

\begin{proposition}\label{pro:heinonen1441}
  In a uniformly perfect metric space $(X,d)$ a set $E$ is fat (thin) if and only if it is quasisymmetrically thick (null).
  \end{proposition}
  \begin{proof}

For any metric space $Y$ and all $\mu\in\mathcal{D}(Y)$, it follows that if $f\colon X\rightarrow Y$ is a quasisymmetric homeomorphism, then the pullback measure $\nu(E)=\mu(f(E))$ is a doubling measure on $X$.

Indeed, if $x\in X$, $0<r<\diam(X)$ and $z\in B(x,2 r)\setminus B(x,r)$, then it follows from the definition of quasisymmetry that
\begin{align*}
&B_Y\left(f(x),d_Y(f(x),f(z))/\eta(2)\right)\subset f\left(B_X(x,r)\right),\\
&f\left(B_X(x,2r)\right)\subset B_Y\left(f(x),\eta(2)d_Y(f(x),f(z))\right).
\end{align*}
Thus $\nu$ is doubling with doubling constant $C^{2\log_2(\eta(2))+1}$, where $C$ is the doubling constant of $\mu$.

Applying this observation when $Y$ is $q$-regular and $\mu=\HH^q$ implies that $\HH^q(f(E))>0$ whenever $E\subset X$ is fat and $\HH^q(f(E))=0$ whenever $E$ is thin. Thus fat sets are quasisymmettrically thick and thin sets are quasisymmetrically null.

Let $\mu$ be a doubling measure on $X$. Then it follows from \cite[Proposition 14.14]{Heinonen} that there is $0<q<\infty$, a metric space $Y$, and a quasisymmetric homeomorphism $f\colon X\rightarrow Y$ such that $\nu(f(E))=\mu(E)$ is Ahlfors $q$-regular (and thus comparable to $\HH^q$) on $Y$. Thus, if $\HH^q(f(E))>0$ then $\mu(E)>0$ and if $\HH^q(f(E))=0$ then $\mu(E)=0$. Applying this for all $\mu\in\mathcal{D}(X)$ implies that quasisymmetrically thick sets are fat and that quasisymmetrically null sets are thin.
\end{proof}

In light of the Proposition \ref{pro:heinonen1441},
our results can be viewed as a step towards understanding
the analogs of the one
dimensional results on quasisymmetrically thick sets in metric spaces. The question as
to what extent such analogs can hold was raised by Heinonen in
\cite[14.42]{Heinonen} and \cite[Open problem 1.18]{Heinonen2003}. See
also \cite[Chapter 16]{DavidSemmes1997}.

The paper is organized as follows.
Section \ref{section:preli} contains some basic facts concerning
doubling measures, uniform perfectness, and fat and thin sets. In
Section \ref{section:thick}, we give a sufficient condition for a
Cantor type set being fat. Namely, we prove that an $(\alpha_n)$-thick
set on a uniformly perfect metric space is fat if $(\alpha_n)_{n=1}^\infty\in\ell^0$.
Theorem
\ref{thm:gradulause} is proved in Section \ref{section:mainthm}. At the end, we present
a couple of examples and open questions related to our results.

\section{Preliminaries}\label{section:preli}

Let us begin with some notation. If there is no danger of
misunderstanding, we obey the notational convention that $C$ denotes
a doubling constant of a measure and $D$ is a uniform perfectness
constant of a metric space in question. Also, whenever we talk about a
ball $B$, it is to be understood that
the center and radius have been fixed (in general these might not be
uniquely determined just by the set $B$).
We also write $c B = B(x,c r)$ for $c > 0$.
By a maximal packing of balls of radius $r>0$ in a set $A$, we mean a collection of disjoint balls
\begin{equation*}
 \mathcal{B} = \{B(x,r) : x \in A\}
\end{equation*}
such that for any $y \in A$ there exists $B \in \mathcal{B}$, for which $B \cap B(y,r) \neq \emptyset$.
For any $\zeta > 0$ and $S \subset X$, the open $\zeta$-neighborhood of the set $S$ is denoted by
\begin{equation*}
 S(\zeta)=\{x \in X:d(x,S)<\zeta\}.
\end{equation*}

We recall a few easy facts about doubling measures and uniformly
perfect metric spaces. These can be found for example in \cite[4.16,
13.1]{Heinonen}.
The first estimate is a direct consequence of the doubling condition:
Let $(X,d)$ be a metric space and $\mu\in\mathcal{D}(X)$.
Then, for all bounded $A\subset X$ with $\mu(A)>0$, $x\in A$ and $0<r < \diam(A)$, we have
\begin{equation}\label{eqn:lemH31}
 \frac{\mu(B(x,r))}{\mu(A)} \ge 2^{-s} \left(\frac{r}{\diam (A)}  \right)^{s},
\end{equation}
where $s = \log_2 C > 0$.
The second estimate is for $\mu\in\mathcal{D}(X)$, when $X$ is uniformly perfect:
There exist constants $\Lambda \ge 1$ and $t > 0$, so that
\begin{equation}\label{eqn:tokajuttu}
 \frac{\mu(B(x,r))}{\mu(B(x,R))} \le \Lambda \left( \frac{r}{R} \right)^{t},
\end{equation}
for all $x \in X$ and $0 < r \le R < \diam(X)$.

Combining the previous two estimates we have the following.

\begin{lem}\label{le:seuraus}
  If $X$ is a bounded uniformly perfect metric space and
  $\mu\in\mathcal{D}(X)$, then there exist constants $0<\lambda,t,\Lambda,s<\infty$
  depending on $\mu$,
 so that
 \begin{equation}\label{eqn:lemmaseurausgeq}
  \lambda r^{s} \leq \mu(B(x,r)) \leq\Lambda r^{t},
 \end{equation}
 for all $x \in X$ and $0<r< \diam(X)$.
\end{lem}

Our next proposition allows us to restrict our considerations to the
case when the space is bounded. We denote by $\mu|_A$ the restriction
of a measure $\mu$ onto a measurable set $A$ and by $\overline{A}$ the
closure of a set $A$.

\begin{proposition}\label{pro:restricttobounded}
 Let $(X,d)$ be a
 metric space
 and let $K\subset X$ be bounded. Then
 there exists a
 bounded closed set $A \subset X$ so that $K\subset A$
 and $\mu|_A \in \mathcal{D}(A)$
 for every $\mu \in \mathcal{D}(X)$.
 Moreover, if the space $(X,d)$ is uniformly perfect, so is the space $(A,d|_{A\times A})$. 
\end{proposition}
\begin{proof}
Choose $k_0\in\mathbb{Z}$ such that
$2^{-k_0-1}\le\diam(K)<2^{-k_0}$. Define $A_{k_0}=K$ and
by recursion
 \[
  A_k = \bigcup_{x \in A_{k-1}} B(x,2^{-k})
 \]
 for all integers $k > k_0$. With these sets, we define
 \[
  A = \overline{\bigcup_{i=k_0}^\infty A_i}.
 \]
 Let $\mu \in \mathcal{D}(X)$, $x \in A$ and $r>0$. Because $\diam(A)
 < 2^{-k_0+2}$, we may assume that $r < 2^{-k_0+2}$. Let $k \in
 \mathbb{Z}$
 be such that $2^{-k} \le r < 2^{-k+1}$.
 Now $d(x,A_{k+2})\le 2^{-k-2}$ so there exists some $y \in A_{k+2} \cap B(x,2^{-k-1})$.
 Using the fact that
 \[
  B(y,2^{-k-3}) \subset A \cap B(x,2^{-k}),
 \]
 we can estimate
 \begin{align*}
  \mu|_A(U(x,2r)) & \le \mu(U(x,2^{-k+2})) \le \mu(U(y,2^{-k+3})) \le C^6 \mu(U(y,2^{-k-3})) \\
                  & \le C^6 \mu|_A(U(x,2^{-k})) \le C^6 \mu|_A(U(x,r)),
 \end{align*}
 where $C$ is the doubling constant of $\mu$. This estimate shows that
 $u|_A \in \mathcal{D}(A)$ since it does not matter in the definition
 of doubling measures whether we use open or closed balls.

 Let us then assume that $(X,d)$ is uniformly perfect with a constant $D$.
 We need to show that the set $A$ as a metric space equipped with the original metric is also uniformly perfect with some constant.
 Take $x \in X$ and $r > 0$. Let $k$ and $y$ be chosen as above.
  If $d(x,A_{k+2})>2^{-k-3}$ it follows that
 \[
  y \in A_{k+2} \cap B(x,r) \setminus B(x,2^{-k-3}) \subset A \cap B(x,r) \setminus B(x,\tfrac{r}{16}).
 \]
 On the other hand, if $d(x,A_{k+2}) \leq 2^{-k-3}$ it follows that $x \in A_{k+3}$ which implies that $B(x,2^{-k-4}) \subset A_{k+4}$ and so $B(x,\tfrac{r}{32}) \subset A$.
 Thus $A$ is uniformly perfect with a constant $32 D$.
 
\end{proof}

One may wonder how rich are the families of fat and thin sets in
general. It is clear that sets with nonempty interior are always fat
and that countable sets not including isolated points of the space are thin. A priori it is
not clear whether these could be the only fat (resp. thin) sets of the
metric space. Also, from the point of view of Theorem
\ref{thm:gradulause} it is reasonable to ask whether there are sets
that are not fat nor thin.

\begin{proposition}\label{prop:existence}
Let $X$ be a complete and doubling metric space such that the set of isolated
points of $X$ is not dense in $X$. Then there is a fat set $E\subset X$
which is nowhere dense and a Cantor set (i.e. an uncountable perfect set
without isolated points) $F\subset X$ which is thin. Also, there is a
compact $G\subset X$ which is not fat nor thin.
\end{proposition}

\begin{proof}

To prove the existence of $E$, we apply a theorem by Saksman from \cite{Saks} that provides a link between nowhere dense fat sets and the existence of doubling measures (see also Remark \ref{remarks:sec3}).

By our assumption on the space there exists a point $x_0 \in X$ and a radius $r>0$
so that the open ball $U(x_0,r)$ does not contain any isolated points.
Consider $K=U(x,r/6)$ and let $A$ be the set constructed in (the proof of) Proposition \ref{pro:restricttobounded}. It is easy to check that $A\subset U(x_0,r)$.
It now follows form \cite[Theorem 5]{Saks} that as a nonempty metric space without isolated points, $A$
contains a dense open subset $G$, for which $\mathcal{D}(G)=\emptyset$.
Now the set $E = A \setminus G$ is clearly nowhere dense in $X$.
  If $\mu\in\mathcal{D}(X)$, it follows by Proposition \ref{pro:restricttobounded} that also $\mu|_A\in\mathcal{D}(A)$.
If it would happen that
$\mu(E) = 0$, this would imply also that $\mu|_G$ is a doubling measure on $G$. Since $\mathcal{D}(G)=\emptyset$, this is impossible. This shows that
 $E$ has to be fat.

To construct
$F$, we may assume that $\diam(X)\le\tfrac14$ and that $X$ has no
isolated points. We consider the
following nested structure on $X$ (see e.g. \cite{KaeSuoRaj} for how this
can be constructed). We let $\{Q_{k,i}\,:\,k\in\N, i\in\{1,\ldots,n_k\}\}$ be Borel
sets such that
\begin{enumerate}[(i)]
\item $X=\bigcup_{i=1}^{n_k}Q_{k,i}$ for every $k\in\N$.
\item $Q_{k,i}\cap Q_{m,j}=\emptyset$ or $Q_{k,i}\subset Q_{m,j}$ when
  $k,m\in\N$, $k\ge m$, $i\in\{1,\ldots,n_k\}$ and $j\in\{1,\ldots,n_m\}$.
\item for every $k\in\N$ and $i\in\{1,\ldots,n_k\}$, there exists a point
  $x_{k,i}\in X$ so that $U(x_{k,i},4^{-k})\subset Q_{k,i}\subset
  B(x_{k,i},4^{-k+2})$.
\end{enumerate}
Given a sequence $0<\alpha_n<1$, we define a Cantor set $F$ using
the following procedure: We first let $F_0=X=Q_{1,1}$. Assume that
$F_n$ has been defined and that it is a pairwise disjoint union
 \[
  F_n=\bigcup_{j}Q^j
 \]
 of elements
 \[
  \{Q^j\}\subset\bigcup_{k\in\N, i\in\{1,\ldots, n_k\}}\{Q_{k,i}\}.
 \]

For each $Q^j$, we choose $\widetilde{Q}^j\in\bigcup_{i}\{Q_{k,i}\}$ such that $\widetilde{Q}^j\subset Q^j$, where
$k=k(Q^j)\in\N$ is chosen such that
\[
4^{-k+3}<\alpha_n\diam(Q^j)\le4^{-k+4}.
\]
Then we define
\[
F_{n+1}=\bigcup_{j}Q^j\setminus\widetilde{Q}^j
\]
and the resulting Cantor set as
\[
F=\bigcap_{n\in\N}\overline{F_n}.
\]

To prove that $F$ is thin, let $\mu$ be a doubling measure on $X$. Consider $F_n=\bigcup_{j}Q^j$ as above. Then it follows from \eqref{eqn:lemH31} that for some $c>0$ and $0<s<\infty$ (depending on $\mu$ but not on $n)$, we have $\mu(\widetilde{Q}^j)\ge c\alpha_{n}^s\mu(Q^j)$ for all $Q^j$ and consequently
\[\mu(F_{n+1})\le(1-c\alpha_{n}^s)\mu(F_n)\]
for all $n$. If $\sum_n \alpha_{n}^s=\infty$, this implies $\mu(F)=0$. Moreover, if we choose $(\alpha_n)\notin\bigcup_{0<p<\infty}\ell^p$, then this holds simultaneously for all $\mu\in\mathcal{D}(X)$ thus implying that $F$ is thin.

The last claim follows from the fact that there are two
doubling measures $\mu$ and $\nu$ on $X$ that are singular with
respect to each other. In the compact case, this was proved in
\cite{KaufmanWu}. For the general, unbounded case, see e.g. \cite{KaeSuoRaj}.
\end{proof}

\begin{rem}
\emph{1.}
The main property of the set $F$ above is that it is \emph{$(c\alpha_n)$-porous} for the defining sequence $(\alpha_n)_n$. It is a known result that this implies the thinnes of $F$
provided $(\alpha_n)\notin\bigcup_{0<p<\infty}\ell^p$. See \cite{Wu}, \cite{Lehto2010}, or \cite{CsoSuom}.

\emph{2.} 
It is maybe worthwhile to compare the notions of being thin or fat
with other concepts of size such as dimension. In
such a comparison, we easily observe that thin and fat sets
cannot be completely characterised in terms of dimension. For
instance, given a complete doubling metric space $X$ and
$\epsilon>0$, there is always a set $E\subset X$ with (Hausdorff
and packing) dimension at most $\epsilon$ which is not thin. See
\cite[Theorem 4.1]{KaeSuoRaj}. On the other hand, there can be sets
with full dimension that are not fat; Consider, for instance, a
Lebesgue null set
$E\subset\R$ of dimension one.
\end{rem}

\section{A sufficient condition for fatness}\label{section:thick}

The aim of this section is to give a sufficient condition for a Cantor
type set being fat, provided that the ``complementary holes'' of the
Cantor set are small enough. A result similar to Theorem
\ref{thm:toinenjuttu} was proved by Staples and Ward
\cite[Theorem 1.4]{StaplesWard} on the real-line and, in fact, the
proof presented here is based on similar iterative use of
\eqref{eqn:tokajuttu} as their proof.
However, the definition of $(\alpha_n)$-thickness we present here is a
relaxed version of the previous definition also on the real-line,
since we do not require the covering sets to be totally disjoint, but
instead assume a uniform bound on the overlaps. It is easy to check
that a set $E\subset\R$ that is $(\alpha_n)$-thick as defined in
\cite{StaplesWard} is
a countable union of
sets that are $(\alpha_n)$-thick in the sense of the
Definition \ref{def:thick} below.

\begin{defin}\label{def:thick}
 We say that a set $E \subset X$ is \emph{$(\alpha_n)$-thick}, for a
 sequence $0<\alpha_n<1$, if there exist constants $N \in \N$, $0< c
 \leq 1$ and finite or countable collections of Borel sets
 $\mathcal{I}_n = \{I_{n,j} \subset X \}$ and $\mathcal{J}_n = \{ J_{n,j} \subset I_{n,j} \}$ for all $n \in \N$ with the properties:
 \begin{enumerate}[(i)]
   \item $E_0=\bigcup_{n\in\N}\bigcup_{I\in\mathcal{I}_n}I$ is bounded.
   \item For all $n \in \N$, each $x \in X$ belongs to at most $N$ different sets $I_{n,j}$.
   \item $c \diam(J_{n,j}) \leq \alpha_n \diam(I_{n,j})$ holds for all $J_{n,j}$.
   \item For all $I_{n,j}$ there exists some $x_{n,j} \in I_{n,j}$ so that
         \begin{equation*}
          B(x_{n,j},\delta_{n,j}) \subset I_{n,j} \setminus \bigcup_{k=1}^{n-1} \bigcup_{i} J_{k,i},
         \end{equation*}
         where $0<\delta_{n,j} = c \diam(I_{n,j})$.
   \item \[
          E_0 \setminus \bigcup_{n=1}^{\infty} \bigcup_j J_{n,j} \subset E.
         \]
 \end{enumerate}
\end{defin}

\begin{thm}\label{thm:toinenjuttu}
  Let $(X,d)$ be a uniformly perfect metric space. If $ (\alpha_n)_{n=1}^\infty \in \ell^0$
  and $E \subset X$ is $(\alpha_n)$-thick, then $E$ is fat.
\end{thm}

\begin{proof}

Let $\mu \in \mathcal{D}(X)$.
 For $n\in\N$, denote
 \[
  E_n = E_{n-1} \setminus \bigcup_{k=1}^{n-1} \bigcup_{j} J_{k,j},
 \]
 $r_{n,j} = \diam(J_{n,j})$, $R_{n,j}=\diam(I_{n,j})$ and pick
 $m\in\N$ such that $2^{-m}<c$, where $c$ is from Definition \ref{def:thick}.
Let $N_0<n \in \N$ and pick $y_{n,j} \in J_{n,j}$ for each $J_{n,j}$.
 Now it follows from \eqref{eqn:tokajuttu} that there exist constants $0<t,C_1<\infty$ such that
 \begin{equation*}
  \mu(J_{n,j}) \le C_1 \left(\frac{r_{n,j}}{R_{n,j}}\right)^{t}\mu(B(y_{n,j},R_{n,j}))
 \end{equation*}
 By the doubling property, there is also $1\le C_2<\infty$ such that
   $\mu(B(y_{n,j},R_{n,j})) \le C_2 \mu(B(x_{n,j},R_{n,j}))$.
 We combine the previous two estimates
and recall the definitions of
 $r_{n,j}$, $R_{n,j}$, $\delta_{n,j}$ and $m$ to get
 \begin{align*}
  \mu(J_{n,j})
  \le c^{-t}C_1 C_2  \alpha_{n}^{t} \mu(B(x_{n,j},R_{n,j}))
  & \le c^{-t}C_1 C_2  \alpha_{n}^{t} \mu(B(x_{n,j},2^{m}\delta_{n,j})) \\
  & \le C_3 \alpha_{n}^{t} \mu(B(x_{n,j},\delta_{n,j})),
 \end{align*}
 where $C_3 = c^{-t}C_1 C_2 C^{m}$.
Now
 \begin{equation*}
  \mu\left( \bigcup_{j} J_{n,j} \right) \leq C_3 \alpha_n^t \sum_{j} \mu( B(x_{n,j},\delta_{n,j}) )
  \leq C_3 N \alpha_n^t \mu(E_{n}),
 \end{equation*}
and thus
 \begin{align*}
  \mu(E_{n+1})  = \mu\left(X \setminus \bigcup_{k=1}^{n} \bigcup_{j} J_{k,j} \right)
  & \ge \mu\left(E_{n}\right) - \mu\left( \bigcup_{j} J_{n,j} \right)\\
  & \ge \left( 1- C_3 N \alpha_n^{t} \right) \mu\left( E_{n} \right).
 \end{align*}
 Applying this for $k=N_0,\ldots,n-1$ yields
 \begin{equation*}
  \mu(E_{n}) \geq\prod_{k=N_0}^{n-1} \left( 1 - C_3 N \alpha_n^{t}
  \right)\mu(E_{N_0}),
 \end{equation*}
 for all $n > N_0$.
 Finally, since $\sum_{n=1}^\infty\alpha_n^t < \infty$, there exists $N_0$ such that
 $(1-C_3 N \alpha_n^t) > 0$ for all $n \geq N_0$, and that
 \[
  \prod_{n=N_0}^{\infty} \left( 1 - C_3 N \alpha_{n}^{t} \right) > 0,
 \]
 which implies $\mu(E)>0$.
\end{proof}

\begin{rem}\label{remarks:sec3}
\emph{1.} It is an interesting fact, that if $E\subset
X$ is fat and nowhere dense, then the set $X\setminus E$ cannot carry
nontrivial doubling measures. For if it did, these measures could be
extended as zero to the set $E$ which would then contradict the set
$E$ being fat in $X$.
Thus, from Theorem \ref{thm:toinenjuttu} we can conclude that if $(\alpha_n)_{n=1}^\infty \in \ell^0$ and
$E\subset X$ is nowhere dense and $(\alpha_n)$-thick, then the
set $X \setminus E$ does not support doubling measures. We remark that
Saksman \cite{Saks} used essentially this idea to show that
any nonempty metric space without isolated points has an open dense
subset that carries no doubling measures.

\emph{2.} For sufficiently regular Cantor sets $C\subset\R$, it can be
shown that the condition $(\alpha_n)\in\ell^0$ is also necessary for
the set being fat. See \cite{BuHaMac}, \cite{Hanetal2009},
\cite{CsoSuom}. It is an open question whether such results
can be generalised to higher dimensional Euclidean spaces or to more
general metric spaces. Similar questions can be asked about
$(\alpha_n)$-porous sets and the condition $(\alpha_n)_{n=1}^\infty\notin\bigcup_{0<p<\infty}\ell^p$.
\end{rem}

\section{Proof of Theorem \ref{thm:gradulause}}\label{section:mainthm}

We start the proof of Theorem \ref{thm:gradulause} with a couple of lemmas.
In the case of the Lebesgue measure on $\R^n$,
the measure of an annulus $B(x,r+\epsilon) \setminus B(x,r)$ behaves like
$r^{n-1}\epsilon$ and can thus be pushed small by choosing
$\epsilon>0$ small enough. Although such an estimate does not
generalise directly for doubling measures in uniformly perfect
metric spaces, we still get an estimate that is good enough for our purposes.
This is the part of the proof which requires the balls to be closed.

\begin{lem}\label{le:pullistus}
 Let $X$ be a bounded uniformly perfect metric space
 and let $\mu\in\mathcal{D}(X)$.
 Let further $\left( B_i \right)_{i=1}^\infty$ be a sequence of closed balls in $X$,
 for which
 \[
  (\diam(B_i))_{i=1}^\infty \in \ell^0\quad \text{and}
  \quad \mu\left(X\setminus \bigcup_{i=1}^{\infty} B_i\right) = \epsilon > 0,
 \]
 and which is ordered so that $\diam(B_i) \geq \diam(B_{i+1})$ for all $i\in\N$.
 Then there exist constants $N_0\in\N$ and $0<Q<\infty$ such that
 \begin{equation}
   \mu\left( \bigcup_{i=1}^{N} B_{i}(2N^{-Q}) \right) < \mu(X) - \frac{\epsilon}{2},
   \label{eqn:pullistuslemma}
 \end{equation}
 whenever $N\geq N_0$.
\end{lem}

\begin{proof}
 Let us choose $0<Q<\infty$ such that
$\Lambda(D+2)^{t} 2^{1-tQ} < \frac{\epsilon}{6}$,
 where $\Lambda$ and $t$
 are from Lemma \ref{le:seuraus}. For any $N\in \N$, let $\Ll(N) = \max\{i:\diam(B_i) \geq N^{-Q} \}$.
 Since for any ball $B_i=B(x_i,r_i)$ we have $r_i \leq \diam(B_i) D$, we get
 \begin{align*}
  &\sum_{i=\Ll(N) +1 }^{N}  \mu(B_i(2 N^{-Q})) \le
  \sum_{i=\Ll(N)+1}^{N} \Lambda \left( r_i+2 N^{-Q} \right)^{t} \\
  & \leq \sum_{i=\Ll(N)+1}^{N} \Lambda \left( \diam(B_i) D+2 N^{-Q} \right)^{t}
    \leq \sum_{i=\Ll(N)+1}^{N} \Lambda \left( (D+2) N^{-Q} \right)^{t} \\
  & \leq
     \Lambda (D+2)^{t} N^{1-tQ} < \frac{\epsilon}{6},
 \end{align*}
 for all $N \geq 2$ by the choice of $Q$. If $N<\Ll(N)+1$ we interpret
 the above sums as zero.

 Now choose $N_1\in\N$ such that
 \begin{align*}
   &\sum_{i=N_1}^{\Ll(N)} \mu(B_i(2 N^{-Q})) \leq \sum_{i=N_0}^{\Ll(N)} \mu(8 B_i) \leq C^3 \sum_{i=N_1}^{\Ll(N)} \mu(B_i) \\
   &\leq C^3\sum_{i=N_1}^{\Ll(N)} \Lambda r_{i}^{t} \leq  C^3 \Lambda D^{t} \sum_{i=N_1}^{\infty} \diam(B_i)^{t} < \frac{\epsilon}{6},
 \end{align*}
 for all $N\in \N$.
 This can be done because $\sum_{i=1}^{\infty} \diam(B_i)^{t} <
 \infty$. If $\Ll(N)<N_1$ we again interpret the sums as zero.

 Finally choose $N_0 > N_1$ such that
 \begin{equation*}
  \mu\left(  \bigcup_{i=1}^{N_1}B_i (2 N_0^{-Q}) \right)
  \leq \mu\left( \bigcup_{i=1}^{N_1} B_i \right) + \frac{\epsilon}{6}.
 \end{equation*}
 Now we compute the measure of the neighborhood in three parts and get
 \begin{align*}
  \mu\left( \bigcup_{i=1}^{N} B_{i}(2 N^{-Q}) \right)&  \le \mu\left( \bigcup_{i=1}^{N_1}B_i(2 N^{-Q}) \right)
 + \sum_{i=N_1}^{\Ll(N)} B_i(2 N^{-Q}) + \sum_{i=\Ll(N)+1}^{N} B_i(2 N^{-Q})\\
 & <  \mu\left( \bigcup_{i=1}^{N_1}B_i \right) + \frac{\epsilon}{6} + \frac{\epsilon}{6} + \frac{\epsilon}{6}
 \le \mu(X) - \frac{\epsilon}{2},
 \end{align*}
   when $N \geq N_0$.
 \end{proof}

Our next lemma gives an estimate on the size of the largest sub-ball
of the $N$:th
approximation $E_N=X\setminus\bigcup_{i=1}^N B_i$ of the cut-out set,
provided that the diameters of the cut-out balls have the required
decay and that the cut-out set is not thin.

\begin{lem}\label{le:suurinpallojaljella}
 Let $X$ be a bounded uniformly perfect metric space and
 $\mu\in\mathcal{D}(X)$. Let further
 $\left( B_i \right)_{i=1}^\infty$ be a sequence of closed balls in $X$ for which
 \[
  (\diam(B_i))_{i=1}^\infty \in \ell^0\quad \text{and}
  \quad \mu\left(X\setminus \bigcup_{i=1}^{\infty} B_i\right) > 0,
 \]
 and which
 are ordered so that $\diam(B_i)\ge \diam(B_{i+1})$ for all $i$. Then
 there exists $0<R<\infty$ and $N_1 \in \N$ such that for all $N\geq
 N_1$ there exists a ball $G_N\subset X \setminus \bigcup_{i=1}^{N}
 B_i$ such that
  $\diam(G_N) \geq N^{-R}$.
\end{lem}

\begin{proof}
 Let us first choose $Q$ and $N_0$ as in Lemma \ref{le:pullistus}.
 For any $N \geq N_0$, let $\mathcal{B}_N$ be a maximal packing of open
 balls of radius $\frac{1}{2 N^Q}$ in the whole of $X$. Since
 $X \subset \bigcup_{\mathcal{B}_N} 2 B,$
 we find a disjoint cover $\mathcal{P}_N=\{P_1,P_2,\ldots\}$ of the space $X$, such that
 for any $P \in \mathcal{P}_N$, there exists $B \in \mathcal{B}_N$
 such that $B \subset P \subset 2 B$. Indeed, we can simply define
 $P_1=2B_1\setminus\bigcup_{i\ge2}B_i$ and then inductively
 \[
 P_n=2B_n\setminus\left(\bigcup_{j=1}^{n-1}P_j\cup\bigcup_{i>n}B_i\right),
 \]
 for $n\ge 2$.

 Denote $\epsilon = \mu\left(X\setminus \bigcup_{i=1}^{\infty} B_i\right)$. According to Lemma \ref{le:pullistus}, we have
 \begin{equation*}
  \sum_{\substack{P \in \mathcal{P}_N\\P \cap \bigcup_{i=1}^N B_i \neq \emptyset }}  \mu(P)
  \leq \mu\left( \bigcup_{i=1}^N B_i(2 N^{-Q})\right)
  < \mu(X) - \frac{\epsilon}{2}
  = \sum_{P \in \mathcal{P}_N} \mu(P)-\frac{\epsilon}{2},
 \end{equation*}
 and so there exist $P\in\mathcal{P}_N$ and $B(x,1/(2N^Q))$ with
 \[
  B\left(x,1/(2 N^Q)\right) \subset P \subset X \setminus \bigcup_{i=1}^{N} B_{i}.
 \]
 For this ball
 $\diam(B) \geq  \frac{1}{2 D N^Q} $. By choosing $R>Q$, we find $N_1 \ge N_0$ such that
 \begin{equation*}
  \diam(G_N) \geq 1/(2 D N^Q) > N^{-R},
 \end{equation*}
 for all $N \ge N_1$.
\end{proof}

We also employ an easy algebraic lemma:

\begin{lem}\label{le:vikalemma}
  For each $\epsilon>0$ and $\delta > \gamma + 1$, where $\gamma > 0$, there exists $M \in \N$ such that
  \begin{equation*}
    \sum_{m=N}^{\infty} \frac{1}{m^{\delta}} < \frac{\epsilon}{N^{\gamma}},
  \end{equation*}
 for all $N \geq M$.
\end{lem}

\begin{proof}
 It is enough to observe that
 \[
  N^\gamma\sum_{m=N}^{\infty} \frac{1}{m^{\delta}} < N^\gamma\int_{N-1}^\infty \frac{1}{x^\delta}dx
   = \frac{N^\gamma}{\delta-1}(N-1)^{1-\delta} \longrightarrow 0
 \]
 as $N \longrightarrow \infty$.
\end{proof}

Now we are ready to prove Theorem \ref{thm:gradulause}.

\begin{proof}[Proof of Theorem \ref{thm:gradulause}]
 Assume that $E$ is not thin so that $\mu(E)>0$ for some $\mu \in
 \mathcal{D}(X)$. Let
 $\nu\in\mathcal{D}(X)$. We have to show that also $\nu(E)>0$.

We first reduce the problem to the case when $X$ is bounded. Pick
$x_0\in X$ and choose $T>0$ so large that $\mu(B(x_0,T)\cap E)>0$ and $\diam(B_i)<T$ for all
$i\in\N$. Moreover, let $A\supset B(x_0,2T)$ be as in Proposition
\ref{pro:restricttobounded}. For each $B_i\in\{B_i\}_{i=1}^\infty$ for
which
\begin{equation}\label{eq:reunapallo}
B_i\cap
A\neq\emptyset\neq B_i\setminus A,
\end{equation}
we pick $y_i\in B_i\cap A$ and
define $\widetilde{B}_i=B(y_i,\diam(B_i))\cap A$. Let $(B'_i)_{i=1}^\infty$
consist of all the original cut-out balls that are completely inside
$A$ and of the balls $\widetilde{B}_i$ for those $B_i$ that satisfy
\eqref{eq:reunapallo}. Now, if we replace $X$ by $A$, $\mu$ by
$\mu|_A$, and $(B_i)_i$ by $(B'_i)_i$ we have reduced the
problem to the case when $X$ is bounded. Observe that
$A\setminus\bigcup_{i}B'_i\subset E$,
$\mu(A\setminus\bigcup_{i}B'_i)>0$ by the choice of $T$, and
that $(\diam(B'_i))_i\in\ell^0$ since
$\diam(\widetilde{B}_i)\leq 2\diam(B_i)$ for each $i \in \N$.

 From now on, we assume that $X$ is bounded. By reordering, we may
 also assume that the balls
   $B_i$ satisfy $\diam(B_i) \ge \diam(B_{i+1})$ for all $i \in \N$.
 Because $\mu(E)>0$, we can apply Lemma \ref{le:suurinpallojaljella}.
 Let $N_1$ and $R$ be as in Lemma \ref{le:suurinpallojaljella} and take
 $N>N_1$. Let $G_N$ be the ball given by Lemma
 \ref{le:suurinpallojaljella} for which $G_N\subset
 X\setminus\bigcup_{i=1}^NB_i$ and $\diam(G_N)\ge N^{-R}$. Now we estimate
 \begin{equation}\label{eqn:lopunalku}
  \nu(E) \ge \nu(G_N \cap E) = \nu\left(G_N \setminus \bigcup_{m=N+1}^{\infty}B_m\right)
  = \nu(G_{N}) - \nu\left(\bigcup_{m=N+1}^{\infty}B_m \cap G_N\right).
 \end{equation}
 Using Lemmas \ref{le:seuraus} and \ref{le:suurinpallojaljella}, we find constants $0<C_1,s<\infty$ such that
 \begin{eqnarray}\label{eqn:Gnkoko}
  \nu(G_N) \geq C_1 \diam(G_N)^{s}   \geq C_1 N^{-Rs}.
 \end{eqnarray}
 Also, by Lemma \ref{le:seuraus} we have constants $0<t,C_2<\infty$,
 such that, for all $p>0$,
 \begin{equation}\label{eqn:Gnleikkauskoko}
\begin{split}
  \nu\left(\bigcup_{m=N+1}^{\infty}B_m \cap G_N\right) \leq \sum_{m=N+1}^{\infty} \nu(B_m \cap G_N)
\leq \sum_{m=N+1}^{\infty} C_2\diam(B_m)^{t}\\
   \leq C_2 \sum_{m=N}^{\infty} \left(\frac{1}{m}\sum_{k=1}^m
    \diam(B_k)^p \right)^{t/p} \leq C_2 \sum_{m=N}^{\infty} \left(\frac{c_p}{m}\right)^{t/p},
\end{split}
\end{equation}
 where $c_p = \sum_{k=1}^\infty \diam(B_k)^p$.

 Finally, let us choose $p$ such that $0 < p < \frac{t}{(Rs +1)}$
 and use Lemma \ref{le:vikalemma} with
 $\epsilon =  C_1 c_p^{-t/p}C_2^{-1}$, $\delta = \frac{t}{p}$ and $\gamma = Rs$.
 We thus find $M$ such that for any $N > \max\{N_0,M\}$, according to Lemma
 \ref{le:vikalemma}, \eqref{eqn:lopunalku}, \eqref{eqn:Gnleikkauskoko} and \eqref{eqn:Gnkoko}, we have
 \begin{equation*}
  \nu(E) \geq \frac{C_1}{N^{Rs}} - c_p^{t/p} C_2
  \sum_{m=N}^{\infty} \left( \frac{1}{m} \right)^{t/p} >0.
 \end{equation*}
 Since $\nu\in\mathcal{D}(X)$ was arbitrary, we conclude that $E$ is fat.
\end{proof}

\section{Examples and open problems}\label{section:examples}

The first example illustrates the conclusion of Theorem
\ref{thm:gradulause} in the simplest case when $X$ is an interval on
the real-line.

\begin{exa}
 Let $(\alpha_n)_{n=1}^\infty\in\ell^0$ and $X = [0,T]$ with the Euclidean
 distance, where $T=\sum_{n=1}^\infty\alpha_n$. Let
 $I_1,I_2,\ldots\subset X$ be closed intervals with length $|I_n|=\alpha_n$.
Now, simply by looking at the Lebesgue measure of
$E=X\setminus\bigcup_{n=1}^\infty I_n$,
Theorem \ref{thm:gradulause} implies that the set $E$ is thin if and only if the interiors
 of the intervals $I_i$ are pairwise disjoint. Observe that also in
 this case, $E$ is often a Cantor type set.
\end{exa}

The next example shows why the balls in Theorem \ref{thm:gradulause} have to be closed.
See \cite[Example 6.3]{Rajala} for another use of a similar construction.

\begin{exa}\label{ex:openfails}
 Take $(\beta_n)_{n=1}^\infty \in \ell^s \setminus \ell^r$ for some $0<r<s<1$,
 and consider the middle interval Cantor set
 $\mathcal{C}(\beta_n)\subset[0,1]$ induced by the sequence
 $(\beta_n)_{n=1}^\infty$ and defined as
 \[
  \mathcal{C}(\beta_n) = \bigcap_{k \in \N} \bigcup_{i=1}^{2^k} I_{k,i},
 \]
 where $I_{1,1}=[0,1]$ and the intervals $I_{k+1,i}$ are achieved by removing
 an open interval of length $\beta_k |I_{k,i}|$ from the middle of each interval $I_{k,i}$.
As proved by Buckley, Hanson, and MacManus
 \cite[Theorem 0.4]{BuHaMac}, the choice $(\beta_n)_{n=1}^\infty \in \ell^s \setminus
 \ell^r$ implies that $\mathcal{C}(\beta_n)$ is neither thin nor fat
 in $[0,1]$.

 Next we define a mapping
 \[
  f \colon [0,1] \to \R^2 \colon x \mapsto \left(x,\frac{d_E(x,\mathcal{C}(\beta_n))}{2 d_E(1/2,\mathcal{C}(\beta_n))}\right),
 \]
 where $d_E$ is the standard Euclidean metric in $\R$. Let $X$ be the graph of $f$
 and equip it with the maximum metric:
\[d\left((x,f(x)),(y,f(y)\right)=\max\{|x-y|,|f(x)-f(y)|\}.\]
 Now measures on $[0,1]$ can be mapped onto measures on $X$ under
 $x\mapsto (x,f(x))$ and it is easy to check that doubling
 measures of $[0,1]$ are transformed onto doubling measures of $X$.
 Thus, we conclude that the set
 \[
  \partial B\left( (1/2,1/2), 1/2 \right) = \{(x,0):x \in
  \mathcal{C}(\beta_n) \} \subset X
 \]
 is neither thin nor fat in $X$.
 Finally, this set can be realized as a cut-out set corresponding
 to the sequence $(\alpha_n)=(2^{-n})$ with open cut-out balls
$U_i = U\left( (1/2,1/2), 2^{-i} \right)$. Indeed,
$X \setminus \bigcup_{i=1}^\infty U_i = X\setminus U_1=\partial B\left( (1/2,1/2),1/2
\right)$ and we observe that Theorem \ref{thm:gradulause} does not
hold if open balls are used in place of closed balls.
\end{exa}

\begin{rem}
We remark that in concrete situations Theorem \ref{thm:toinenjuttu} is sometimes more useful than
Theorem \ref{thm:gradulause}. For instance, if
$\mathcal{C}(\beta_n)$ is as in Example \ref{ex:openfails}, we see that the
assumptions of Theorem \ref{thm:toinenjuttu} are satisfied precisely
when $(\beta_n)_{n=1}^\infty\in\ell^0$ whereas for the natural cut-out balls
(i.e. the complementary intervals of the set $C(\beta_n)$) the
assumptions of Theorem \ref{thm:gradulause} are not satisfied even for
$\beta_n=2^{-n}$. See \cite[Example 2.2]{StaplesWard}. However, the
strength of Theorem \ref{thm:gradulause} (compared to Theorem
\ref{thm:toinenjuttu}, for instance) is that there are no geometric
assumptions on the location of the cut-out balls. As the above example
indicates, in general the assumption $(\beta_n)_{n=1}^\infty\in\ell^0$ seems to be
too strong and thus it seems reasonable to pose the following question:
\end{rem}

\begin{question}
Given a doubling uniformly perfect metric space $X$, does there exist
$p=p(X)>0$ such
that the conclusion of Theorem \ref{thm:gradulause} holds for cut-out
sets for which $(\diam(B_i))\in\ell^p$. In particular, does any value
$0<p<1$ suite as $p(\R)$ or $p([0,1])$?
\end{question}

The following example shows that necessarily $p([0,1])<1$.

\begin{exa}
Let $\mu$ be any doubling measure on $[0,1]$ which is singular with
respect to the Lebesgue measure. Then we can find pairwise disjoint closed
intervals $I_1,I_2,\ldots\subset[0,1]$ such that for $\alpha_n=|I_n|$,
we have $\sum_{n=1}^\infty\alpha_n=1$ but $\mu([0,1]\setminus\bigcup_{n=1}^\infty I_n)>0$.
\end{exa}

 Our final example shows that the assumption on uniform perfectness in
 our results is really needed.

\begin{exa}\label{ex:nuperf}
We construct a Cantor set $\mathcal{C}(\beta_n)\subset\R$ (which is
not uniformly perfect) as in Example \ref{ex:openfails}
and show that
the conclusion of Theorem \ref{thm:gradulause} does not hold for
$X=\mathcal{C}(\beta_n)$ equipped with the Euclidean metric.
The numbers $\frac{1}{3} < \beta_n<1$ will be
determined later. In fact, it would be possible to choose
$(\beta_n)_n$ such that any compact $E\subset\mathcal{C}(\beta_n)$ is
a cut-out set for some $(B_i)_i$ with
$(\diam(B_i))_{i=1}^\infty\in\ell^0$ and then use the general
existence result of Proposition \ref{prop:existence}. For the sake of
concreteness, we give a more detailed example below.

 Let $m_j=\lfloor\log_2(j+1)\rfloor$, where $\lfloor\cdot\rfloor$
 denotes integer part. Define $k_1=1$ and
 $k_{j+1}=k_j+m_j$, for $j\in\N$. We construct a set $E
 \subset X$ inductively, using the sequences $(k_j)_{j=1}^\infty$ and $(m_j)_{j=1}^\infty$ as
 follows:
We first
 remove from $\mathcal{C}(\beta_n)$ the left construction interval of level
 level $k_2=2$ (i.e. $[0,\tfrac12(1-\beta_1)]$).
 At step $j>1$, we have removed from $\mathcal{C}(\beta_n)$ some of the
 level $k_j$ construction intervals. From each of the remaining level
 $k_j$ intervals, we then remove the left-most level $k_{j+1}$ interval.
 We let $E\subset\mathcal{C}(\beta_n)$ be the remaining set.

 For each $0<p<1$, consider the binomial measure $\mu_p$ on $\mathcal{C}(\beta_n)$ that
 is obtained by assigning the weight $p$ to the left interval and $1-p$ to the right
interval on each step in the construction of $\mathcal{C}(\beta_n)$. Now the $\mu_p$ measure of $E$ is
 \[
  \mu_p(E) = \prod_{i=0}^\infty(1-p^{m_i})
 \]
 which is zero if and only if
 \[
  \sum_{i=0}^\infty p^{m_i} = \sum_{i=0}^\infty (j+1)^{1/\log_p(2)}=  \infty
 \]
 which is the case if and only if $p \ge \frac{1}{2}$. Thus, for example, $\mu_{1/3}(E)>0$ and $\mu_{2/3}(E) = 0$.
 Next we observe that since $\beta_n>\frac{1}{3}$ for all $n$,
 we have $\mu_p\in\mathcal{D}(X)$ for all $0<p<1$. We thus conclude that
 the set $E$ is neither fat nor thin in $\mathcal{C}(\beta_n)$.

 Let $(I_n)_{n\in\N}$ be the construction intervals removed from
 $\mathcal{C}(\beta_n)$. Since $\beta_n>\tfrac13$, the sets
 $B_n=I_n\cap\mathcal{C}(\beta_n)$ are balls as subsets of
 $\mathcal{C}(\beta_n)$. Moreover, since we are free to choose each $\beta_n$ as close to
 $1$ as we wish, we can easily guarantee that $(\diam(B_i))_{i=1}^\infty\in\ell^0$.
 Notice that by Theorem \ref{thm:gradulause}, we
 know that
 the resulting Cantor set $\mathcal{C}(\beta_n)$ is not uniformly perfect, i.e. that $\limsup_{n\rightarrow\infty}\beta_n=1$.

The set $E\subset\mathcal{C}(\beta_n)$ constructed above is clearly $(1-\beta_n)_n$--thick (Also any singleton in $\mathcal{C}(\beta_n)$
  is $(1-\beta_n)$--thick) and
thus choosing $(1-\beta_n)_n\in\ell^0$ implies that the assumption on
uniform perfectness is needed also in Theorem \ref{thm:toinenjuttu}.
\end{exa}

\subsection*{Acknowledgements}
T. Ojala was supported by the Vilho, Yrj\"o and Kalle V\"ais\"al\"a fund,
T. Rajala by the European Project ERC AdG *GeMeThNES* and
V. Suomala by the Academy of Finland project \#126976.

\bibliographystyle{alpha}

\bibliography{thinfat}

\end{document}